\documentclass{amsart}

\usepackage[latin1]{inputenc}
\usepackage[english]{babel}
\usepackage{indentfirst}
\usepackage{amssymb}
\usepackage{amsthm}

\newcommand{\N}{\mathbb{N}}
\newcommand{\erre}{\mathbb{R}}
\newcommand{\sub}{\subseteq}

\newtheorem{theo}{Theorem}[section]
\newtheorem{lem}[theo]{Lemma}%[section]
\newtheorem{pro}[theo]{Proposition}%[section]
\newtheorem{cor}[theo]{Corollary}%[theo]
\newtheorem{defi}[theo]{Definition}%[section]
\newtheorem{rem}[theo]{Remark}
\newtheorem{exa}[theo]{Example}

\numberwithin{equation}{section}

\title{A note on summability in Banach spaces}

\author{Jos\'{e} Rodr\'{i}guez}
\address{Dpto. de Matem\'{a}ticas\\E.T.S. de Ingenieros Industriales de Albacete\\
Universidad de Castilla-La Mancha\\ 02071 Albacete\\ Spain} \email{jose.rodriguezruiz@uclm.es}


\thanks{The research was supported by grants PID2021-122126NB-C32 
(funded by MCIN/AEI/10.13039/501100011033 and ``ERDF A way of making Europe'', EU) and 
21955/PI/22 (funded by {\em Fundaci\'on S\'eneca - ACyT Regi\'{o}n de Murcia}).}

\begin{document}

\begin{abstract}
Let $Z$ and~$X$ be Banach spaces. Suppose that $X$ is Asplund. Let $\mathcal{M}$ be a bounded set of operators from~$Z$ to~$X$
with the following property: a bounded sequence $(z_n)_{n\in \N}$ in~$Z$ is weakly null if, 
for each $M \in \mathcal{M}$, the sequence $(M(z_n))_{n\in \N}$ is weakly null.
Let $(z_n)_{n\in \N}$ be a sequence in~$Z$ such that: (a)~for each $n\in \N$, the set $\{M(z_n):M\in \mathcal{M}\}$ is relatively norm compact;
(b)~for each sequence $(M_n)_{n\in \N}$ in~$\mathcal{M}$, the series $\sum_{n=1}^\infty M_n(z_n)$ is weakly unconditionally Cauchy. 
We prove that if $T\in \mathcal{M}$ is Dunford-Pettis and $\inf_{n\in \N}\|T(z_n)\|\|z_n\|^{-1}>0$, then
the series $\sum_{n=1}^\infty T(z_n)$ is absolutely convergent. As an application, we provide another proof of the fact that
a countably additive vector measure taking values in an Asplund Banach space has finite variation whenever its integration operator is Dunford-Pettis.
\end{abstract}

\maketitle

\section{Introduction}

Let $X$ be a Banach space, let $(\Omega,\Sigma)$ be a measurable space and let $\nu:\Sigma\to X$
be a countably additive vector measure. A $\Sigma$-measurable function $f:\Omega \to \mathbb{R}$ is said to be {\em $\nu$-integrable} if: 
(a)~$f$ is $|x^*\nu|$-integrable for all $x^*\in X^*$; (b)~for each $A\in \Sigma$ there is $\int_A f \, d\nu\in X$
such that $x^*\left ( \int_A f \, d\nu \right)=\int_A f\,d(x^*\nu)$ for all $x^*\in X^*$.
By identifying $\nu$-a.e. equal functions, the set $L_1(\nu)$ of all (equivalence classes of) $\nu$-integrable functions is a Banach lattice with the $\nu$-a.e. order and the norm
$$
	\|f\|_{L_1(\nu)}:=\sup_{x^*\in B_{X^*}}\int_\Omega |f|\,d|x^*\nu|.
$$
We refer to~\cite{oka-alt} for basic information on these spaces, which play a relevant role in Banach lattices and operator theory.
The {\em integration operator} of~$\nu$ is the (norm one) operator $I_\nu:L_1(\nu)\to X$ defined by
$$
	I_\nu(f):=\int_\Omega f \, d\nu \quad\text{for all $f\in L_1(\nu)$}.
$$
Certain properties of~$I_\nu$ have strong consequences on the structure of~$L_1(\nu)$. For instance,
$\nu$ has finite variation and the inclusion operator $\iota_\nu: L_1(|\nu|) \to L_1(\nu)$ is a lattice-isomorphism in each of the following cases:
\begin{itemize}
\item[(i)] $I_\nu$ is compact, \cite[Theorem~1]{oka-alt2} (cf. \cite[Theorem~2.2]{oka-alt4} and \cite[Theorem~3.3]{cal-alt-5});
\item[(ii)] $I_\nu$ is absolutely $p$-summing for some $1\leq p <\infty$, \cite[Theorem~2.2]{oka-alt3};
\item[(iii)] $I_\nu$ is Dunford-Pettis and Asplund, \cite[Theorem~3.3]{rod15}.
\end{itemize}
Note that case (iii) generalizes both (i) and (ii) because weakly compact operators are Asplund. The proof of (iii) given in \cite{rod15} (cf. \cite[Section~3.3]{rod23}) is based on the
Davis-Figiel-Johnson-Pe{\l}czy\'{n}ski factorization procedure and the following result obtained in \cite[Theorem~1.3]{cal-alt-5}:

\begin{theo}\label{theo:CRS}
Let $X$ be a Banach space, let $(\Omega,\Sigma)$ be a measurable space and let $\nu:\Sigma\to X$
be a countably additive vector measure. If $I_\nu$ is Dunford-Pettis and $X$~is Asplund, then $\nu$ has finite variation.
\end{theo}

The particular case of Theorem~\ref{theo:CRS} when $X$ has an unconditional Schauder basis and no subspace isomorphic to~$\ell_1$
had been proved earlier in~\cite[Theorem~1.2]{oka-alt3}. The question of whether the statement of Theorem~\ref{theo:CRS}
holds for arbitrary Banach spaces not containing subspaces isomorphic to~$\ell_1$ seems to be still open.

In this note we elaborate an abstract framework that allows to provide a simpler proof of Theorem~\ref{theo:CRS}. The following concept
will be important along this way. Given two Banach spaces $Z$ and~$X$, we denote by $\mathcal{L}(Z,X)$ the Banach space of all operators
from~$Z$ to~$X$, equipped with the operator norm.

\begin{defi}\label{defi:RainwaterOperators}
Let $Z$ and $X$ be Banach spaces. We say that a set $\mathcal{M} \sub \mathcal{L}(Z,X)$
has the {\em Rainwater property} if the following holds: a bounded sequence $(z_n)_{n\in \N}$ in~$Z$ is weakly null if, 
for each $M\in \mathcal{M}$, the sequence $(M(z_n))_{n\in \N}$ is weakly null.
\end{defi}

The Rainwater-Simons theorem (see, e.g., \cite[Theorem~3.134]{fab-ultimo}) states that, for an arbitrary Banach space~$Z$, 
any James boundary of~$Z$ has the Rainwater property (with $X=\mathbb R$). More generally,
James boundaries are (I)-generating, \cite[Theorem~2.3]{fon-lin}, and all (I)-generating sets have the Rainwater property, see~\cite{nyg4}.

The main result of this note is the following:
 
\begin{theo}\label{theo:main}
Let $Z$ and~$X$ be Banach spaces. Suppose that $X$ is Asplund. Let $\mathcal{M}$ be a bounded subset of $\mathcal{L}(Z,X)$
having the Rainwater property. Let $(z_n)_{n\in \N}$ be a sequence in~$Z$ such that: 
\begin{enumerate}
\item[(a)] for each $n\in \N$, the set $\{M(z_n):M\in \mathcal{M}\}$ is relatively norm compact;
\item[(b)] for each sequence $(M_n)_{n\in \N}$ in~$\mathcal{M}$, the series $\sum_{n=1}^\infty M_n(z_n)$ is weakly unconditionally Cauchy. 
\end{enumerate}
Let $T\in \mathcal{M}$ such that:
\begin{enumerate}
\item[(c)] $T$ is Dunford-Pettis; 
\item[(d)] $\inf_{n\in \N}\|T(z_n)\|\|z_n\|^{-1}>0$.
\end{enumerate}
Then the series $\sum_{n=1}^\infty T(z_n)$ is absolutely convergent.
\end{theo}

The paper is organized as follows.
Section~\ref{section:main} is devoted to proving Theorem~\ref{theo:main}. In 
Section~\ref{section:L1} we focus on the $L_1$ space of a vector measure and we get
Theorem~\ref{theo:CRS} as an application of Theorem~\ref{theo:main}.

\subsubsection*{Terminology}

All our Banach spaces are real. By an {\em operator} we mean a continuous linear map between Banach spaces. 
An operator is called {\em Dunford-Pettis} if it maps weakly null sequences to norm null ones.
By a {\em subspace} of a Banach space we mean a closed linear subspace.
Let $Z$ be a Banach space. We denote its norm by $\|\cdot\|_Z$ or simply~$\|\cdot\|$. 
Given a set $C \sub Z$, we write $\|C\|:=\sup\{\|z\|:z\in C\}$.
The closed unit ball of~$Z$ is denoted by~$B_Z$. The subspace of~$Z$ generated by a set $H \sub Z$ is denoted by $\overline{{\rm span}}(H)$.
We write $Z^*$ for the dual of~$Z$. A set $B \sub B_{Z^*}$ is said to be a {\em James boundary} of~$Z$
if for every $z\in Z$ there is $z^*\in B$ such that $\|z\|=z^*(z)$. The space~$Z$ is said to be {\em Asplund}
if every separable subspace of~$Z$ has separable dual.

\section{Main result}\label{section:main}

Let $X$ be a Banach space with a Schauder basis $(e_n)_{n\in \N}$. For each $k\in \N$, we have an operator
$P_k:X\to X$ defined by $P_k(x):=\sum_{n=1}^k e_n^*(x)e_n$ for all $x\in X$, where
$(e_n^*)_{n\in \N}$ is the sequence in~$X^*$ of biorthogonal functionals associated with~$(e_n)_{n\in \N}$.
The operators of this form are called the {\em partial sum operators} on~$X$ associated with~$(e_n)_{n\in \N}$. They satisfy $\sup_{k\in \N}\|P_k\|<\infty$. 

The following lemma uses some ideas of the proof of \cite[Lemma~3.4]{rod23}. 

\begin{lem}\label{lem:projection}
Let $X$ be a Banach space with a Schauder basis and let $(P_k)_{k\in \N}$ be the associated sequence of partial sum operators on~$X$. 
Write $\alpha:=\sup_{k\in \N}\|P_k\|$. Let $(K_n)_{n\in\N}$ be a sequence of relatively norm compact subsets of~$X$ 
and let $(x_n)_{n\in \N}$ be a sequence in~$X$ such that $x_n\in K_n$ for all $n\in \N$. Suppose that:
\begin{itemize}
\item[(a)] the series $\sum_{n=1}^\infty x_n$ is not absolutely convergent;
\item[(b)] $\sum_{n=1}^\infty\|P_k(K_n)\|<\infty$ for every $k\in \N$.
\end{itemize}
Then there exist two strictly increasing sequences $(k_j)_{j\in \N}$ and $(l_j)_{j\in \N}$ in~$\N$ such that, if $w_j \in \|x_{l_{j}}\|^{-1} K_{l_{j}}$
for all $j\in \N$, then:
\begin{enumerate}
\item[(i)] $\|w_j-(P_{k_{j+1}}-P_{k_{j}})(w_j)\|\leq 2^{-j}$ for every $j\in \N$; 
\item[(ii)] $\|w_j-w_{j'}\| \geq \alpha^{-1}\|w_j\|- 2^{-j}$ whenever $j'>j$.
\end{enumerate}
\end{lem}
\begin{proof} We can assume without loss of generality that $x_n\neq 0$ for all $n\in \N$.
Write $Q_k:={\rm id}_X-P_k$ for all $k\in \N$, where ${\rm id}_X$ stands for the identity operator on~$X$.
Since $\|Q_k\|\leq 1+\alpha$ for all $k\in \N$ and $\|Q_k(x)\|\to 0$
as $k\to \infty$ for every $x\in X$, the sequence of operators $(Q_k)_{k\in \N}$ converges to~$0$ uniformly on each relatively norm compact subset of~$X$.

We will construct by induction strictly increasing sequences $(k_j)_{j\in \N}$ and $(\tilde{l}_j)_{j\in \N}$ in~$\N$
in such a way that, for each $j\in \N$, we have 
$$
	{\rm (c)} \ \ \left\|P_{k_{j}}(K_{\tilde{l}_{j+1}})\right\| \leq  \frac{\|x_{\tilde{l}_{j+1}}\|}{2^{j+1}} \qquad\mbox{and}\qquad
	{\rm (d)} \ \ \left\|Q_{k_{j}}(K_{\tilde{l}_j})\right\| \leq \frac{\|x_{\tilde{l}_j}\|}{2^j}.
$$
Set $\tilde{l}_1:=1$ and choose $k_1\in \N$ such that $\|Q_{k_1}(K_1)\|\leq \frac{1}{2}\|x_1\|$. Suppose that
$k_N,\tilde{l}_N\in \N$ are already chosen for some $N\in \N$. By (a) and~(b), there is $\tilde{l}_{N+1}\in \N$ with $\tilde{l}_{N+1}>\tilde{l}_N$ such that
$$
	  \left\|P_{k_N}(K_{\tilde{l}_{N+1}})\right\| \leq \frac{\|x_{\tilde{l}_{N+1}}\|}{2^{N+1}}.
$$
Now, we take $k_{N+1}\in \N$ with $k_{N+1}>k_N$ such that 
$$
	\left\|Q_{k_{N+1}}(K_{\tilde{l}_{N+1}})\right\| \leq \frac{\|x_{\tilde{l}_{N+1}}\|}{2^{N+1}}.
$$
This finishes the construction of $(k_j)_{j\in \N}$ and $(\tilde{l}_j)_{j\in \N}$. 

Define $l_j:=\tilde{l}_{j+1}$ for all $j\in \N$. Take $(z_j)_{j\in \N} \in \prod_{j\in \N}K_{l_j}$
and, for each $j\in \N$, define $w_j:=\|x_{l_{j}}\|^{-1}z_j$. Then 
\begin{eqnarray*}
	\|w_j-(P_{k_{j+1}}-P_{k_{j}})(w_j)\|&=&\|Q_{k_{j+1}}(w_j) + P_{k_{j}}(w_j)\| \\
	&\leq & \|Q_{k_{j+1}}(w_j)\| + \|P_{k_{j}}(w_j)\| \\ &\stackrel{\text{(c)\&(d)}}{\leq}& \frac{1}{2^{j+1}}+\frac{1}{2^{j+1}}=\frac{1}{2^j}
\end{eqnarray*}
for every $j\in \N$. This proves~(i).

To check property (ii), take $j'>j$ in~$\N$. Then
\begin{eqnarray*}
	\|w_j-w_{j'}\| & \geq & \alpha^{-1}\| P_{k_{j+1}}(w_j-w_{j'})\| \\ &= & \alpha^{-1}\|w_j-Q_{k_{j+1}}(w_j)-P_{k_{j+1}}(w_{j'})\| \\
	& \geq &  \alpha^{-1}\Big(\|w_j\|-\|Q_{k_{j+1}}(w_j)\|-\|P_{k_{j+1}}(w_{j'})\|\Big) \\ &= & 
	\alpha^{-1}\Big(\|w_j\|-\|Q_{k_{j+1}}(w_j)\|-\|P_{k_{j+1}}(P_{k_{j'}}(w_{j'}))\|\Big) \\
	& \geq & \alpha^{-1}\Big(\|w_j\|-\|Q_{k_{j+1}}(w_j)\|-\alpha\|P_{k_{j'}}(w_{j'})\|\Big) \\
	& \stackrel{(\alpha\geq 1)}{\geq} & \alpha^{-1}\|w_j\|-\|Q_{k_{j+1}}(w_j)\|-\|P_{k_{j'}}(w_{j'})\| \\
	& \stackrel{\text{(c)\&(d)}}{\geq} & \alpha^{-1}\|w_j\|-\frac{1}{2^{j+1}} -\frac{1}{2^{j'+1}} \ \geq \ \alpha^{-1}\|w_j\|-\frac{1}{2^j}.
\end{eqnarray*}
The proof is finished.
\end{proof}

\begin{cor}\label{theo:discrete-compact}
Let $X$ be a Banach space. Let $(x_n)_{n\in\N}$ be a sequence in~$X$ such that $\sum_{n=1}^\infty x_n$ is weakly unconditionally Cauchy
and $\{\|x_n\|^{-1} x_n: \, n\in \N, \, x_n\neq 0\}$ is relatively norm compact. Then $\sum_{n=1}^\infty x_n$ is absolutely convergent.
\end{cor}
\begin{proof}
The subspace $\overline{{\rm span}}(\{x_n:n\in \N\})\sub X$  is separable, so it embeds isometrically into the Banach space $C([0,1])$.
Hence, we can assume without loss of generality that $X=C([0,1])$. Since this space has a Schauder basis, the conclusion follows 
from Lemma~\ref{lem:projection}(ii) by taking $K_n:=\{x_n\}$ for all $n\in \N$. Indeed, if 
$(P_k)_{k\in \N}$ is the sequence of partial sum operators on~$C([0,1])$ associated with a given Schauder basis, then 
for each $k\in \N$ the series $\sum_{n=1}^\infty P_k(x_n)$ is absolutely convergent, because
it is weakly unconditionally Cauchy and $P_k(X)$ is finite-dimensional.
\end{proof}

Let $X$ be a Banach space with a Schauder basis $(e_n)_{n\in \N}$.
By a {\em block sequence with respect to $(e_n)_{n\in \N}$} we mean 
a sequence $(x_j)_{j\in \N}$ in~$X$ for which there exist a sequence $(a_n)_{n\in \mathbb N}$ in~$\mathbb R$ and a sequence $(I_j)_{j\in \N}$ of non-empty finite subsets of~$\N$ 
such that $\max(I_j)<\min(I_{j+1})$ and $x_j=\sum_{n\in I_j}a_ne_n$ for all $j\in \N$. Recall that the Schauder basis
$(e_n)_{n\in \N}$ is said to be {\em shrinking} if its sequence 
of biorthogonal functionals $(e_n^*)_{n\in \N}$ satisfies $X^*=\overline{{\rm span}}(\{e_n^*:n\in \N\})$.

We can now prove our main result.

\begin{proof}[Proof of Theorem~\ref{theo:main}] 
Clearly, we can suppose that $\|M\|\leq 1$ for every $M\in \mathcal{M}$. 

Let us consider the subspace $Z_0:=\overline{{\rm span}}(\{z_n:n\in \N\}) \sub Z$.
The set of restrictions $\{M|_{Z_0}:M\in \mathcal{M}\} \sub B_{\mathcal{L}(Z_0,X)}$ has the Rainwater property and fulfills conditions (a) and~(b). 
Obviously, the restriction $T|_{Z_0}$ also satisfies conditions~(c) and~(d).
The subspace
$$
	X_0:=\overline{{\rm span}}\left(\bigcup_{n\in \N} \{M(z_n): \, M\in \mathcal{M}\}\right) \sub X
$$
is separable (thanks to~(a)) and we have $M(Z_0) \sub X_0$ for every $M\in \mathcal{M}$. Since $X$ is Asplund and $X_0$ is separable, $X_0^*$ is separable.
Therefore, we can assume without loss of generality that $X^*$ is separable. 

A result of Zippin~\cite{zip} (cf. \cite[Chapter~5]{dod2}) states that every Banach space with separable dual embeds isomorphically
into a Banach space with a shrinking Schauder basis. Therefore, we can assume further that $X$ has a shrinking Schauder basis, say~$(e_n)_{n\in \N}$.
Let $(P_k)_{k\in \N}$ be the sequence of partial sum operators on~$X$ associated with~$(e_n)_{n\in \N}$. 

For each $n\in \N$, write $x_n:=T(z_n)\in X$ and consider the relatively norm compact set 
$$
	K_n:=\{M(z_n):\, M\in \mathcal{M}\} \sub X.
$$
Observe that for each $k\in \N$ we have $\sum_{n=1}^\infty\|P_k(K_n)\|<\infty$. Indeed, for 
every $n\in \N$ we choose $M_n\in \mathcal{M}$ such that 
\begin{equation}\label{eqn:abs}
	\|P_k(K_n)\|\leq \|P_k(M_n(z_n))\|+\frac{1}{2^{n}}.
\end{equation}
Since $\sum_{n=1}^\infty M_n(z_n)$ is weakly unconditionally Cauchy (by condition~(b)) and $P_k(X)$ is finite-dimensional, the series
$\sum_{n=1}^\infty P_k(M_n(z_n))$ is absolutely convergent and so inequality~\eqref{eqn:abs} yields $\sum_{n=1}^\infty\|P_k(K_n)\|<\infty$, as claimed.

Suppose, by contradiction, that $\sum_{n=1}^\infty x_n$ is not absolutely convergent and apply Lemma~\ref{lem:projection}. 
Let $(k_{j})_{j\in\N}$ and $(l_{j})_{j\in\N}$
be as in Lemma~\ref{lem:projection}. Define 
$$
	R_{j}:=P_{k_{j+1}}-P_{k_j}\in \mathcal{L}(X,X)
	\quad\mbox{and}\quad 
	u_j:=\|x_{l_{j}}\|^{-1}z_{l_j} \in Z\quad\text{for all $j\in \N$}.
$$ 

Write $\beta:=\inf_{n\in \N}\|T(z_n)\|\|z_n\|^{-1}>0$. Fix $M\in \mathcal{M}$ and define 
$$
	w^M_j:=M(u_j)=\|x_{l_{j}}\|^{-1} M(z_{l_{j}}) \in \|x_{l_{j}}\|^{-1}K_{l_j}
	\quad\text{for all $j\in \N$}.
$$
Note that $\|u_j\| \leq \beta^{-1}$ and so 
$\|w^M_j\| \leq \|M\|\beta^{-1} \leq \beta^{-1}$ for all $j\in \N$. 
Observe that $(R_{j}(w^M_j))_{j\in \N}$ is a block sequence with respect to~$(e_n)_{n\in\N}$ which is bounded, 
because the sequence $(w^M_j)_{j\in \N}$ is bounded and $\|R_j\|\leq 2\sup_{k\in \N}\|P_k\|<\infty$ for all $j\in \N$. 
Since $(e_n)_{n\in \N}$ is shrinking, we deduce that $(R_{j}(w^M_j))_{j\in \N}$ is weakly null
(see, e.g., \cite[Proposition~3.2.7]{alb-kal}). Since
$$
	\|w^M_j-R_j(w^M_j)\|\leq \frac{1}{2^{j}} \quad\text{for all $j\in \N$}
$$
(by part~(i) of Lemma~\ref{lem:projection}), we conclude that $(w^M_j)_{j\in \N}$ is weakly null as well.

 As $M\in \mathcal{M}$ is arbitrary, the Rainwater property of~$\mathcal{M}$ implies 
 that the sequence $(u_j)_{j\in \N}$ is weakly null in~$Z$. This is a contradiction, because 
 $T$ is Dunford-Pettis and $\|T(u_j)\|=1$ for every $j\in \N$.
\end{proof}

A sequence~$(x_n)_{n\in \N}$ in a Banach space is said to be an {\em $\ell_1$-sequence} if it is bounded and there is a constant $c>0$ such that
\begin{equation*}
	\left\| \sum_{n=1}^N a_n x_n \right\|\geq c \sum_{n=1}^N |a_n|
\end{equation*}
for every $N\in \N$ and for all $a_1,\dots,a_N\in \erre$. That is, $(x_n)_{n\in \N}$ is an $\ell_1$-sequence if and only if it is 
a basic sequence which is equivalent to the usual Schauder basis of~$\ell_1$ (see, e.g., \cite[Section~1.3]{alb-kal}).

\begin{cor}\label{cor:unconditional-basis}
Let $Z$ and $X$ be Banach spaces. Suppose that $X$ is Asplund. 
Let $\mathcal{M}$ be a bounded subset of~$\mathcal{L}(Z,X)$ having the Rainwater property.
Let $(e_n)_{n\in \N}$ be a seminormalized basic sequence in~$Z$ such that:
\begin{enumerate}
\item[(a)] for each $n\in \N$, the set $\{M(e_n):M\in \mathcal{M}\}$ is relatively norm compact;
\item[(b)] for each sequence $(a_n)_{n\in \N}$ in~$\mathbb R$ such that the series $\sum_{n=1}^\infty a_n e_n$ is convergent
and for each sequence $(M_n)_{n\in \N}$ in~$\mathcal{M}$, the series $\sum_{n=1}^\infty a_n M_n(e_n)$ is weakly unconditionally Cauchy.
\end{enumerate}
Let $T\in \mathcal{M}$ such that:
\begin{enumerate}
\item[(c)] $T$ is Dunford-Pettis;
\item[(d)] $\inf_{n\in \N}\|T(e_n)\|\|e_n\|^{-1}>0$.
\end{enumerate}
 Then $(e_n)_{n\in \N}$ is an $\ell_1$-sequence.
\end{cor}
\begin{proof} Let $(a_n)_{n\in \N}$ be a sequence in~$\mathbb R$. 
Then $\sum_{n=1}^\infty |a_n|\|e_n\|<\infty$ if (and only if)
the series $\sum_{n=1}^\infty a_n e_n$ is convergent. To check this, we can assume without loss of generality that $a_n\neq 0$ for all $n\in \N$.
Now, Theorem~\ref{theo:main} (applied to $z_n:=a_n e_n$) ensures that if $\sum_{n=1}^\infty a_n e_n$ is convergent, then
we have $\sum_{n=1}^\infty |a_n|\|T(e_n)\|<\infty$ and so $\sum_{n=1}^\infty |a_n|\|e_n\|<\infty$ (by~(d)). 
This shows that $(\|e_n\|^{-1}e_n)_{n\in \N}$ is an $\ell_1$-sequence. Since $(e_n)_{n\in \N}$ is seminormalized, 
it is an $\ell_1$-sequence as well.
\end{proof}

We finish this section with a few remarks on sets of operators having the Rainwater property and some examples. The first one is an immediate consequence of the aforementioned
Rainwater-Simons theorem (see, e.g., \cite[Theorem~3.134]{fab-ultimo}).

\begin{cor}\label{cor:boundary}
Let $Z$ and $X$ be Banach spaces and let $\mathcal{M} \sub B_{\mathcal{L}(Z,X)}$. 
The following statements are equivalent and imply that $\mathcal{M}$ has the Rainwater property:
\begin{enumerate}
\item[(i)] the set $\bigcup_{M\in \mathcal{M}}M^*(B_{X^*}) \sub B_{Z^*}$ is a James boundary of~$Z$;
\item[(ii)] for every $z\in Z$ there is $M\in \mathcal{M}$ such that $\|z\|=\|M(z)\|$.
\end{enumerate}
\end{cor}

\begin{defi}\label{defi:JBproperty}
Let $Z$ and $X$ be Banach spaces. We say that a set $\mathcal{M} \sub B_{\mathcal{L}(Z,X)}$
has the {\em James boundary property} if it satisfies conditions (i)-(ii) of Corollary~\ref{cor:boundary}.
\end{defi}

\begin{exa}\label{exa:unconditional-sum}
Let $X$ be a Banach space and let $E$ be a Banach space with a normalized $1$-unconditional Schauder basis~$(e_n)_{n\in \N}$.
Let $Z$ be the $E$-sum of countably many copies of~$X$, that is, $Z$ is the Banach space of all sequences $(x_n)_{n\in \N}$ in~$X$
such that the series $\sum_{n=1}^\infty\|x_n\| e_n$ converges in~$E$, equipped with the norm
$$
	\left\|(x_n)_{n\in \N}\right\|_{Z}:=\left\|\sum_{n=1}^\infty \|x_n\|\, e_n\right\|_E.
$$
Let $\mathcal{M} \sub B_{\mathcal{L}(Z,X)}$ be the set of all coordinate projections. 
\begin{enumerate}
\item[(i)] If $E^*$ is separable, then $\mathcal{M}$ has the Rainwater property (see, e.g., \cite[Lemma~3.22]{rod20}).
\item[(ii)] If $E=c_0$, then $\mathcal{M}$ has the James boundary property.
\item[(iii)] If $E=\ell_p$ for some $1<p<\infty$, then  $\mathcal{M}$ has the Rainwater property but
fails to have the James boundary property (unless $X=\{0\}$). Indeed, bear in mind that $c_0$
does not contain subspaces isomorphic to~$\ell_p$ (see, e.g., \cite[Corollary~2.1.6]{alb-kal}).
\end{enumerate}
\end{exa}

It is natural to wonder when a single operator has the Rainwater property. An obvious necessary condition is that
such an operator must be injective. In fact:

\begin{rem}\label{rem:kernel}
\rm Let $Z$ and $X$ be Banach spaces and let $\mathcal{M} \sub \mathcal{L}(Z,X)$ be a set  
having the Rainwater property. Then $\bigcap_{M\in \mathcal M}\ker M=\{0\}$. Indeed, if $z\in \bigcap_{M\in \mathcal M}\ker M$, then the Rainwater property of~$\mathcal{M}$
implies that the constant sequence $(z,z,\dots)$ is weakly null in~$Z$, which is equivalent to saying that $z=0$.
\end{rem}

Let $Z$ and $X$ be Banach spaces. An operator $T:Z\to X$ is called {\em tauberian}
if its second adjoint satisfies $(T^{**})^{-1}(X) \sub Z$. This is equivalent to saying that a bounded set $C \sub Z$ is relatively weakly compact if (and only if)
$T(C)$ is relatively weakly compact (see, e.g., \cite[Corollary~2.2.5]{gon-abe}). As a consequence, we have:

\begin{rem}\label{rem:Tauberian}
Let $Z$ and $X$ be Banach spaces and let $T\in \mathcal{L}(Z,X)$ be injective. 
\begin{enumerate}
\item[(a)] If $T$ is tauberian, then $\{T\}$ has the Rainwater property. 
\item[(b)] If $\{T\}$ has the Rainwater property and $Z$ is weakly sequentially complete, then $T$ is tauberian.
\end{enumerate}
\end{rem}

In part~(b) of the previous remark, the additional assumption on $Z$ cannot be dropped in general:

\begin{exa}\label{exa:Tauberian2}
Let $T:c_0\to \ell_1$ be the injective operator defined by 
$$
	T((a_n)_{n\in \N}):=(2^{-n} a_n)_{n\in \N} 
	\quad\text{for all $(a_n)_{n\in \N}\in c_0$}.
$$
Then $\{T\}$ has the Rainwater property, but $T$ is not tauberian. Indeed, any tauberian operator maps the closed unit ball
of the domain space to a closed set (see, e.g., \cite[Theorem~2.1.7]{gon-abe}). However,
$T(B_{c_0})$ is not closed. For instance, it is easy to check that $x=(2^{-n})_{n\in \N} \in \ell_1$ satisfies $x\in \overline{T(B_{c_0})} \setminus T(B_{c_0})$. 
\end{exa}

The previous example is a particular case of a more general construction: 

\begin{pro}\label{pro:singleton}
Let $Z$ and $X$ be Banach spaces and let $\mathcal{M} \sub \mathcal{L}(Z,X)$ be a countable set having the Rainwater property. 
Let $E$ be a Banach space with a normalized $1$-unconditional Schauder basis~$(e_n)_{n\in \N}$ and
let $Y$ be the $E$-sum of countably many copies of~$X$. Then there is an injective operator $T: Z\to Y$
such that $\{T\}$ has the Rainwater property.
\end{pro}
\begin{proof}
Enumerate $\mathcal{M}=\{M_n:n\in \N\}$. If we multiply each $M_n$ by a non-zero constant, the resulting set also has the Rainwater property. So,
we can assume that the series $\sum_{n=1}^\infty \|M_n\| e_n$ converges~$E$. Now, the map $T:Z\to Y$ defined by $T(z):=(M_n(z))_{n\in \N}$
for all $z\in Z$ satisfies the requirements.
\end{proof}

\section{Application to the $L_1$ space of a vector measure}\label{section:L1}

Let $X$ be a Banach space, let $(\Omega,\Sigma)$ be a measurable space and let $\nu:\Sigma\to X$ be a countably additive vector measure.
The variation and semivariation of~$\nu$ are denoted by $|\nu|$ and $\|\nu\|$, respectively. Given $x^*\in X^*$, 
we denote by $x^*\nu:\Sigma\to \mathbb R$ the composition of $\nu$ with~$x^*$ and we denote by~$|x^*\nu|$ its variation.
We say that $A\in \Sigma$ is {\em $\nu$-null} if $\|\nu\|(A)=0$ or, equivalently, $\nu(B)=0$ for every $B \in \Sigma$ contained in~$A$.
The subset of~$\Sigma$ consisting of all $\nu$-null sets is denoted by~$\mathcal{N}(\nu)$.

Every $\nu$-essentially bounded $\Sigma$-measurable function $f:\Omega \to \mathbb{R}$ is $\nu$-integrable. 
By identifying $\nu$-a.e. equal functions, the set $L_\infty(\nu)$ of all (equivalence classes of)
$\nu$-essentially bounded $\Sigma$-measurable functions is a Banach lattice with the $\nu$-a.e. order and the $\nu$-essential supremum norm~$\|\cdot\|_{L_\infty(\nu)}$. 
For each $g\in L_\infty(\nu)$, we denote by $M_g: L_1(\nu) \to X$ the operator defined by
$$
	M_g(f):=\int_\Omega f g \, d\nu
	\quad\text{for all $f\in L_1(\nu)$},
$$
which satisfies $\|M_g\| \leq \|g\|_{L_\infty(\nu)}$. It is known that
\begin{equation}\label{eqn:norming}
	\|f\|_{L_1(\nu)}=\sup_{g\in B_{L_\infty(\nu)}}\|M_g(f)\| \quad
	\text{for all $f\in L_1(\nu)$}
\end{equation}
(see, e.g., \cite[Proposition~3.31]{oka-alt}).

The following lemma can be found in~\cite[Lemma~3.3]{oka-alt3} and \cite[Corollary~4.2]{cal-alt-6}. Note that part~(ii) follows
at once from part~(i) and~\eqref{eqn:norming}. It is worth pointing out that
in~(ii) the set $B_{L_\infty(\nu)}$ can be replaced by its extreme points, that is, the subset $\{\chi_{A}-\chi_{\Omega\setminus A}:A\in \Sigma\}$, 
see \cite[Corollary~2.4]{cal-alt-7}.

\begin{lem}\label{lem:compact}
Let $X$ be a Banach space, let $(\Omega,\Sigma)$ be a measurable space and let $\nu:\Sigma\to X$ be a countably additive vector measure
such that the set $\{\nu(A):A\in \Sigma\}$
is relatively norm compact. Then: 
\begin{enumerate}
\item[(i)] for each $f\in L_1(\nu)$, the set $\{M_g(f): g\in B_{L_\infty(\nu)}\}$ is norm compact;
\item[(ii)] the set $\{M_g: g\in B_{L_\infty(\nu)}\}$ has the James boundary property. 
\end{enumerate}
In particular, $\{M_g: g\in B_{L_\infty(\nu)}\}$ has the Rainwater property.
\end{lem}

\begin{lem}\label{lem:unconditional-basic-sequence-L1}
Let $X$ be a Banach space, let $(\Omega,\Sigma)$ be a measurable space and let $\nu:\Sigma \to X$ be a countably additive vector measure. 
Let $(f_n)_{n\in \N}$ be sequence of pairwise disjoint non-zero elements of~$L_1(\nu)$. Then $(f_n)_{n\in \N}$ is a $1$-unconditional
basic sequence in~$L_1(\nu)$.
\end{lem}
\begin{proof}
It suffices to check that 
\begin{equation}\label{eqn:unconditional}
	\left\|
	\sum_{k=1}^n a_k f_k 
	\right\|_{L_1(\nu)}
	\leq
	\left\|
	\sum_{k=1}^m b_k f_k 
	\right\|_{L_1(\nu)}
\end{equation}
for all sequences $(a_k)_{k\in \N}$ and $(b_k)_{k\in \N}$ in~$\mathbb R$ such that $|a_k|\leq |b_k|$ for every $k\in \N$ and
for all $n\leq m$ in~$\N$ (see, e.g., \cite[Propositions~1.1.9 and 3.1.3]{alb-kal}). Fix $x^*\in B_{X^*}$. Since the $f_k$'s are pairwise disjoint, we have
\begin{multline*}
	\int_\Omega \left|\sum_{k=1}^n a_k f_k \right| \, d|x^*\nu|  =
	\sum_{k=1}^n |a_k|\int_\Omega |f_k| \, d|x^*\nu| \\\leq
	\sum_{k=1}^m |b_k|\int_\Omega |f_k| \, d|x^*\nu|  = \int_\Omega \left|\sum_{k=1}^m b_k f_k \right| \, d|x^*\nu| \leq
	\left\|\sum_{k=1}^m b_k f_k 
	\right\|_{L_1(\nu)}.
\end{multline*}
By taking the supremum when $x^*$ runs over all $B_{X^*}$, we get~\eqref{eqn:unconditional}.
\end{proof}

We can now prove Theorem~\ref{theo:CRS} by using Corollary~\ref{cor:unconditional-basis}.

\begin{proof}[Proof of Theorem~\ref{theo:CRS}]
It suffices to show that $\sum_{n=1}^\infty\|\nu(C_n)\|<\infty$ for every 
sequence $(C_n)_{n\in \N}$ of pairwise disjoint elements of $\Sigma\setminus \mathcal{N}(\nu)$
(see, e.g., \cite[Corollary~2]{nyg-pol2}). 

Fix $\rho>2$ and $n\in \N$. We can take $A_n \in \Sigma \setminus \mathcal{N}(\nu)$ such that $A_n \sub C_n$ and $\rho \|\nu(A_n)\| \geq \|\nu\|(C_n)$.
Define $f_n:=\|\nu\|(C_n)^{-1}\chi_{A_n} \in L_1(\nu)$ and note that 
\begin{equation}\label{eqn:semi}
	\frac{1}{\rho}\leq \frac{\|\nu(A_n)\|}{\|\nu\|(C_n)} \leq \|f_n\|_{L_1(\nu)}=\frac{\|\nu\|(A_n)}{\|\nu\|(C_n)} \leq 1.
\end{equation}
Hence, $(f_n)_{n\in \N}$ is a seminormalized $1$-unconditional basic sequence in~$L_1(\nu)$
(apply Lemma~\ref{lem:unconditional-basic-sequence-L1}). 

We will show that $(f_n)_{n\in \N}$ is an $\ell_1$-sequence via Corollary~\ref{cor:unconditional-basis}
applied to the operator $T:=M_{\chi_\Omega}=I_\nu$ and the family $\mathcal{M}:=\{M_g: g\in B_{L_\infty(\nu)}\}$. 
Since $I_\nu$ is Dunford-Pettis, the set $\{\nu(A):A\in \Sigma\}$ is relatively norm compact (see \cite[Theorem~5.8]{cal-alt-6}, cf. \cite[Proposition~2.6]{rod23}).
Hence, $\mathcal{M}$ has the Rainwater property and condition~(a) of Corollary~\ref{cor:unconditional-basis} holds (apply Lemma~\ref{lem:compact}).
Condition~(d) holds because
$$
	\|I_\nu(f_n)\|_X\|f_n\|_{L_1(\nu)}^{-1} =\frac{\|\nu(A_n)\|}{\|\nu\|(A_n)} \geq \frac{\|\nu(A_n)\|}{\|\nu\|(C_n)}\stackrel{\eqref{eqn:semi}}{\geq} \frac{1}{\rho}
	\quad
	\text{for all $n\in \N$}.
$$
To check condition~(b), let $(a_n)_{n\in \N}$
be a sequence in~$\mathbb R$ such that $\sum_{n=1}^\infty a_n f_n$ is convergent in~$L_1(\nu)$ and  
let $(g_n)_{n\in \N}$ be a sequence in~$B_{L_\infty(\nu)}$. Since the $A_n$'s are pairwise disjoint, we can
find $g \in B_{L_\infty(\nu)}$ such that $g|_{A_n}=g_n|_{A_n}$ for every $n\in \N$. 
Since $(f_n)_{n\in \N}$ is an unconditional basic sequence, $\sum_{n=1}^\infty a_n f_n$ is unconditionally convergent.  
Then the series
$$
	\sum_{n=1}^\infty a_n M_{g_n}(f_n)=\sum_{n=1}^\infty a_n \int_\Omega f_n g_n \, d\nu =
	\sum_{n=1}^\infty a_n \int_\Omega f_n g \, d\nu=	
	\sum_{n=1}^\infty M_g (a_n f_n)
$$
is unconditionally convergent in~$X$ (because $M_g$ is an operator) and, therefore, it is weakly unconditionally Cauchy. 
So, condition~(b) of Corollary~\ref{cor:unconditional-basis} holds too.
From that result it follows that $(f_n)_{n\in \N}$ is an $\ell_1$-sequence.

Let $c>0$ such that 
$$
	\sum_{n=1}^N |a_n|  \leq c \left\|\sum_{n=1}^N a_n f_n\right\|_{L_1(\nu)}
$$
for every $N\in \N$ and for all $a_1,\dots,a_N\in \mathbb R$. The previous inequality applied to $a_n:=\|\nu\|(C_n)$ yields
\begin{multline*}
	\sum_{n=1}^N \|\nu(C_n)\|  \leq  \sum_{n=1}^N \|\nu\|(C_n)  \\ \leq c \left\|\sum_{n=1}^N \chi_{A_n} \right\|_{L_1(\nu)} \
	=  c\left\|\chi_{\bigcup_{n=1}^N A_n} \right\|_{L_1(\nu)} =
	c \left\|\nu\right\|\left(\bigcup_{n=1}^N A_n\right)  \leq  c \left\|\nu\right\|(\Omega)
\end{multline*}
for every $N\in \N$. It follows that $\sum_{n=1}^\infty \|\nu(C_n)\|\leq c\|\nu\|(\Omega)<\infty$, as required.
\end{proof}

\subsection*{Acknowledgements}
The research was supported by grants PID2021-122126NB-C32 
(funded by MCIN/AEI/10.13039/501100011033 and ``ERDF A way of making Europe'', EU) and 
21955/PI/22 (funded by {\em Fundaci\'on S\'eneca - ACyT Regi\'{o}n de Murcia}).

% BIBLIOGRAPHY

\bibliographystyle{amsplain}

\begin{thebibliography}{10}

\bibitem{alb-kal}
F.~Albiac and N.~J. Kalton, \emph{Topics in {B}anach space theory}, Graduate
  Texts in Mathematics, vol. 233, Springer, New York, 2006. 

\bibitem{cal-alt-6}
J.~M. Calabuig, S.~Lajara, J.~Rodr\'{i}guez, and E.~A.  S\'{a}nchez-P\'{e}rez, \emph{Compactness in {$L^1$} of a vector measure},
  Studia Math. \textbf{225} (2014), no.~3, 259--282. 

\bibitem{cal-alt-5}
J.~M. Calabuig, J.~Rodr\'{i}guez, and E.~A. S\'{a}nchez-P\'{e}rez, \emph{On
  completely continuous integration operators of a vector measure}, J. Convex
  Anal. \textbf{21} (2014), no.~3, 811--818. 

\bibitem{cal-alt-7}
J.~M. Calabuig, J.~Rodr\'{i}guez, and E.~A. S\'{a}nchez-P\'{e}rez,
  \emph{Summability in {$L^1$} of a vector measure}, Math. Nachr. \textbf{290}
  (2017), no.~4, 507--519. 

\bibitem{dod2}
P.~Dodos, \emph{Banach spaces and descriptive set theory: selected topics},
  Lecture Notes in Mathematics, vol. 1993, Springer-Verlag, Berlin, 2010.
 
\bibitem{fab-ultimo}
M.~Fabian, P.~Habala, P.~H{\'a}jek, V.~Montesinos, and V.~Zizler, \emph{Banach
  space theory. The basis for linear and nonlinear analysis}, CMS Books in Mathematics, Springer, New York, 2011.

\bibitem{fon-lin}
V.~P. Fonf and J.~Lindenstrauss, \emph{Boundaries and generation of convex
  sets}, Israel J. Math. \textbf{136} (2003), 157--172. 

\bibitem{gon-abe}
M.~Gonz{\'a}lez and A.~Mart{\'{\i}}nez-Abej{\'o}n, \emph{Tauberian operators},
  Operator Theory: Advances and Applications, vol. 194, Birkh\"auser Verlag, Basel, 2010. 

\bibitem{nyg4}
O.~Nygaard, \emph{A remark on {R}ainwater's theorem}, Ann. Math. Inform.
  \textbf{32} (2005), 125--127.

\bibitem{nyg-pol2}
O.~Nygaard and M.~P\~{o}ldvere, \emph{Families of vector measures of uniformly
  bounded variation}, Arch. Math. (Basel) \textbf{88} (2007), no.~1, 57--61.

\bibitem{oka-alt2}
S.~Okada, W.~J. Ricker, and L.~Rodr\'{i}guez-Piazza, \emph{Compactness of
  the integration operator associated with a vector measure}, Studia Math.
  \textbf{150} (2002), no.~2, 133--149.

\bibitem{oka-alt3}
S.~Okada, W.~J. Ricker, and L.~Rodr\'{i}guez-Piazza, \emph{Operator ideal properties of vector measures with finite
  variation}, Studia Math. \textbf{205} (2011), no.~3, 215--249. 

\bibitem{oka-alt4}
S.~Okada, W.~J. Ricker, and L.~Rodr\'{i}guez-Piazza, \emph{Operator ideal properties of the integration map of a vector
  measure}, Indag. Math. (N.S.) \textbf{25} (2014), no.~2, 315--340.

\bibitem{oka-alt}
S.~Okada, W.~J. Ricker, and E.~A. S{\'a}nchez~P{\'e}rez, \emph{Optimal domain
  and integral extension of operators. Acting in function spaces}, Operator Theory: Advances and
  Applications, vol. 180, Birkh\"auser Verlag, Basel, 2008.

\bibitem{rod23}
J.~Rodr\'{i}guez, \emph{{D}unford-{P}ettis type properties in {$L_1$} of a
  vector measure}, preprint, arXiv:2404.05419.

\bibitem{rod15}
J.~Rodr\'{i}guez, \emph{Factorization of vector measures and their integration
  operators}, Colloq. Math. \textbf{144} (2016), no.~1, 115--125.

\bibitem{rod20}
J.~Rodr\'{i}guez, \emph{$\epsilon$-weakly precompact sets in {B}anach spaces}, Studia
  Math. \textbf{262} (2022), no.~3, 327--360.

\bibitem{zip}
M.~Zippin, \emph{Banach spaces with separable duals}, Trans. Amer. Math. Soc.
  \textbf{310} (1988), no.~1, 371--379. 

\end{thebibliography}

\end{document}